\documentclass[a4paper,12pt]{amsart}

\usepackage{enumerate,verbatim,cite} % different enumerate?
\usepackage{pgf,tikz,ifthen}
\usetikzlibrary{arrows}
\usetikzlibrary{calc}

\usepackage[breaklinks=true,colorlinks=true]{hyperref}

\allowdisplaybreaks

\theoremstyle{plain}
\newtheorem{theorem}{Theorem}
\newtheorem*{theorem*}{Theorem}
\newtheorem{lemma}{Lemma}
\newtheorem*{proposition*}{Proposition}
\newtheorem*{corollary*}{Corollary}

\theoremstyle{definition}

\theoremstyle{remark}
\newtheorem*{remark*}{Remark}

\begin{document}

\title{On planar bipartite biregular degree sequences}

\author{Patrick Adams}
\author{Yuri Nikolayevsky}

%\thanks{The first author was supported by the AMSI Vacation Research Scholarship Program.} %The AMSI Vacation Research Scholarship Program is jointly funded by the Department of Education and Training and the Australian Mathematical Sciences Institute

\address{Department of Mathematics and Statistics, La Trobe University, Melbourne, Australia 3086.}
\email{18490186@students.latrobe.edu.au}
\email{Y.Nikolayevsky@latrobe.edu.au}

% 05C10  (1973-now) Planar graphs; geometric and topological aspects of graph theory
% 05C62  (2000-now) Graph representations (geometric and intersection representations, etc.)
% 68R10  (1980-now) Graph theory (including graph drawing)

 % anything else?

%\date{\today}

\begin{abstract}
A pair of sequences of natural numbers is called planar if there exists a simple, bipartite, planar graph for which the given sequences are the degree sequences of its parts. For a pair to be planar, the sums of the sequences have to be equal and Euler's inequality must be satisfied. Pairs that verify these two necessary conditions are called Eulerian. We prove that a pair of constant sequences is planar if and only if it is Eulerian (such pairs can be easily listed) and is different from $(3^5 \, | \, 3^5)$ and $(3^{25} \, | \, 5^{15})$.
\end{abstract}

\maketitle

\section{Introduction}
\label{s:intro}

The theory of graphic sequences dates back to the classical result of P.\,Erd{\H{o}}s and T.\,Gallai from 1960 \cite{EG}, which gives a necessary and sufficient condition for a sequence to be graphic. Recall that a sequence of natural numbers is called \emph{graphic} if there exists a simple graph (no loops, no multiple edges), for which the given sequence is the sequence of degrees of its vertices. In the bipartite case, one speaks of a pair of sequences of natural numbers. Such a pair is called \emph{graphic} if there exists a simple bipartite graph the degree sequences of whose parts are the given two sequences. In 1957, D.\,Gale and H.\,Ryser independently found necessary and sufficient conditions for a pair of sequences to be graphic \cite{Gale,Ryser}.
At present, there are many different proofs and equivalent formulations of the graphicality condition. For most recent results in the bipartite case, we refer the reader to \cite{BR,CMN,MY} and the bibliographies therein. % ref Stacey's thesis? , both for sequences and for pairs of sequences

The picture is much less clear for \emph{planar} graphic sequences, that is, for those which can be realised by a planar simple graph. As the authors say in \cite{SH}, ``The complete characterization of planar graphical degrees sequences is an extremely difficult problem". The classification of planar graphic sequences is known for regular (constant) sequences \cite{HHRT,Owens}, for ``almost regular" sequences, the ones whose maximal and minimal terms differ by no more than $2$ and for some other classes. In the regular case, the characterisation is as follows.

Call a sequence $D=(d_1, d_2, \dots, d_n), \; n \ge 1, d_1 \ge d_2 \ge \dots \ge d_n$, \emph{Eulerian} if the following three conditions are satisfied: the sum $S=\sum_{i=1}^n d_i$ is even, $d_1 < n$, and $m \le 3n-6$ for $n \ge 3$, where $m=\frac12 S$ (Euler's inequality). These conditions are clearly necessary for the sequence to be planar graphic. For $p \ge 1$ and $a \ge 1$ we denote $(a^p)$ the regular sequence consisting of $p$ repeats of $a$.

\begin{theorem}[\cite{HHRT,Owens}] \label{t:reg}
  Let $D=(a^p)$ be a regular sequence.
  \begin{enumerate}[{\rm (1)}]
    \item \label{it:regE}
     The sequence $D$ is Eulerian if and only if $D=(1^{2r}), \, r \ge 1$, or $D=(2^p), \, p \ge 3$, or $D=(3^{2r}), \, r \ge 2$, or $D=(4^p), \, p \ge 6$, or $D=(5^{2r}), \, r \ge 3$.

    \item \label{it:regpg}
    The sequence $D$ is planar graphic if and only if it is Eulerian and is different from $(4^7)$ and $(5^{14})$.
  \end{enumerate}
\end{theorem}

We note that assertion~\eqref{it:regE} is immediate, and then to prove Theorem~\ref{t:reg} one constructs the plane drawings explicitly. The most technically involved part is the proof that the two Eulerian sequences excluded in assertion~\eqref{it:regpg} are indeed not planar graphic.

\smallskip

In this paper, we prove the bipartite counterpart of Theorem~\ref{t:reg}. Following a similar approach, we will define the class of Eulerian sequences which satisfy the ``minimal" necessary conditions, and will then see that these conditions are almost sufficient, with the exception of two cases.

For $1 \le a \le b$ and $p, q \ge 1$, we denote $(a^p \, | \, b^q)$ the pair of sequences in which the first sequence consists of $p$ repeats of number $a$, and the second one, of $q$ repeats of number $b$. We call such a pair a \emph{bipartite biregular pair}. We say that a bipartite biregular pair is \emph{planar graphic} (or simply \emph{planar}), if there exists a simple, bipartite, planar graph for which one part consists of $p$ vertices of degree $a$, while the other part has $q$ vertices of degree $b$. A bipartite biregular pair is said to be \emph{Eulerian}, if $ap=bq$, $q \ge a$, and Euler's inequality $m \le 2(p+q)-4$ is satisfied, where $m=ap=bq$ and $p+q \ge 3$. These conditions are necessary for the pair to be planar graphic. We prove the following.

\begin{theorem} \label{t:bb}
  Let $P=(a^p \, | \, b^q), \; 1 \le a \le b, \; p, q \ge 1$ be a bipartite biregular pair.
  \begin{enumerate}[{\rm (1)}]
    \item \label{it:bbE}
     The pair $P$ is Eulerian if and only if $P=(1^{bq} \, | \, b^q)$, or $P=(2^{br} \, | \, b^{2r}), \, r \ge 1$, or $P=(2^{qr} \, | \, (2r)^q), \, q \ge 3$ is odd, or $P = (3^p \, | \, 3^p), \, p \ge 4$, or $P=(3^{4r} \, | \, 4^{3r}), \, r \ge 2$, or $P = (3^{5r} \, | \, 5^{3r}), \, r \ge 4$.

    \item \label{it:bbpg}
    The pair $P$ is planar graphic if and only if it is Eulerian and is different from $(3^5 \, | \, 3^5)$ and $(3^{25} \, | \, 5^{15})$.
  \end{enumerate}
\end{theorem}

Throughout the paper, a ``drawing" of a graph  always means an embedding. %In the drawings of bipartite graphs, the vertices belonging to the same part are shown in the same colour. % when submit: solid and hollow

\section{Proof of Theorem~\ref{t:bb}}
\label{s:proof}

\eqref{it:bbE} Let $P=(a^p \, | \, b^q), \; 1 \le a \le b, \, 1 \le p,q$. From $ap, bq \le 2 (p+q)-4$ it follows that $a < 4$. Now if $a=1$ or $a=2$, the claim is immediate. If $a=3$, then $q= 3pb^{-1}$ and so $b \le \frac{6p}{p+4} < 6$. This gives $b \in \{3,4,5\}$ and $p \ge \frac{4b}{6-b}$, and then the remaining cases follow.

\eqref{it:bbpg} \underline{Planar cases.} We first show that all the pairs which are claimed to be planar are indeed planar, by explicitly exhibiting plane drawings.

If $a=1$, the pair $P=(1^{bq} \, | \, b^q)$ is always planar: we can take $q$ disjoint copies of $K_{1,b}$. Note that there is only one graph realising this pair and that for $q \ge 2$, it is disconnected. However, in all the cases which follow the plane drawings can be made connected using the following simple procedure.
\begin{figure}[h!]
\tikzstyle{blue} = [circle,draw=blue!50,fill=blue!20,ultra thick,inner sep=0pt,minimum size=3mm]
\tikzstyle{red} = [circle,draw=red!50,fill=red!20,ultra thick,inner sep=0pt,minimum size=3mm]
\begin{tikzpicture}
\node [blue] (v1) at (0,0) {};
\node [red] (v2) at (0,1) {};
\draw (-1,1)--(v2)--(-1,0.8); \draw (-1,1.2)--(v2)--(v1)--(-1,-0.2); \draw (-1,0)--(v1)--(-1,0.2);
\node [red] (v3) at (1,0) {};
\node [blue] (v4) at (1,1) {};
\draw (2,1)--(v4)--(2,0.8); \draw (2,1.2)--(v4)--(v3)--(2,-0.2); \draw (2,0)--(v3)--(2,0.2);
\node [blue] (v5) at (5,0) {};
\node [red] (v6) at (5,1) {};
\node [red] (v7) at (6,0) {};
\node [blue] (v8) at (6,1) {};
\draw (7,1)--(v8)--(7,0.8); \draw (7,1.2)--(v8)--(v6)--(4,1.2); \draw (7,0)--(v7)--(7,0.2);
\draw (4,1)--(v6)--(4,0.8); \draw (7,-0.2)--(v7)--(v5)--(4,-0.2); \draw (4,0)--(v5)--(4,0.2);
\node at (3,0.5) {$\mathbf{\longrightarrow}$};
\end{tikzpicture}
\caption{Reconnecting two edges in the disjoint union of two graphs.}
\label{fig:connect}
\end{figure}
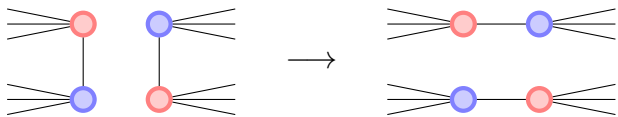

Given the drawings of the disjoint union of two graphs, we place them as on the left in Figure~\ref{fig:connect} and then reconnect two edges as on the right in Figure~\ref{fig:connect}. This results in a plane drawing of a bipartite graph, with the ``correct" degrees of the vertices. If the degrees of all four vertices shown in Figure~\ref{fig:connect} are at least $2$, the resulting graph is connected.

If $a=2$, we consider two cases depending on the parity of $q$. If $q=2r$, we have $P=(2^{br} \, | \, b^{2r})$. This pair is always planar: we can take $r$ copies of $K_{2,b}$ (one can make the resulting drawing connected repeatedly using the procedure as in Figure~\ref{fig:connect}). If $q$ is odd, then $P=(2^{qr} \, | \, (2r)^q), \, q \ge 3$. This pair is also planar. To see that, take $q$ copies of $K_{2,r}$ and in the $i$-th copy, $i=1, \dots, q$, denote $v_i, v_{i+1}$ the two vertices of degree $r$ if $r \ne 2$, (if $r=2$, denote $v_i, v_{i+1}$ any two non-adjacent vertices in the $i$-th copy of $K_{2,2}$). Now identify the points having the same label $v_i$ (where $i$ is computed modulo $q$).

If $a=b=3$ we have $P = (3^p \, | \, 3^p), \, p \ge 4$. For any even $p \ge 4$, the pair $P$ is planar: a planar realisation is the $p$-prism. Replacing a vertex by seven vertices as in Figure~\ref{fig:expansion3} we obtain that if the pair $(3^{p} \, | \, 3^{p})$ is planar, then the pair $(3^{p+3} \, | \, 3^{p+3})$ is also planar. It follows that for all $p \ge 6$ and for $p=4$, the pair $P = (3^p \, | \, 3^p)$ is planar.
\begin{figure}[h!]
\tikzstyle{blue} = [circle,draw=blue!50,fill=blue!20,ultra thick,inner sep=0pt,minimum size=3mm]
\tikzstyle{red} = [circle,draw=red!50,fill=red!20,ultra thick,inner sep=0pt,minimum size=3mm]
\begin{tikzpicture}
\node [blue] (v2) at (-0.2,0) {};
\node (v3) at (-0.2,1.6) {}; %
\node (v1) at (-1.6,-0.8) {};%%
\node (v4) at (1.2,-0.8) {}; %%
\draw  (v1) edge (v2);
\draw  (v2) edge (v3);
\draw  (v2) edge (v4);
\node (v13) at (2.4,-0.8) {};%%
\node (v14) at (5.2,-0.8) {};%%
\node (v6) at (3.8,1.6) {}; %
\node [blue] (v9) at (3.8,0) {};
\node [red] (v7) at (3.2,0.4) {};
\node [red] (v8) at (4.4,0.4) {};
\node [red] (v11) at (3.8,-0.8) {};
\node [blue] (v10) at (3.2,-0.4) {};
\node [blue] (v5) at (3.8,0.8) {};
\node [blue] (v12) at (4.4,-0.4) {};
\draw  (v5) edge (v6);
\draw  (v7) edge (v5);
\draw  (v5) edge (v8);
\draw  (v9) edge (v8);
\draw  (v9) edge (v7);
\draw  (v7) edge (v10);
\draw  (v10) edge (v11);
\draw  (v11) edge (v12);
\draw  (v10) edge (v13);
\draw  (v9) edge (v11);
\draw  (v8) edge (v12);
\draw  (v12) edge (v14);
\node at (1.6,0) {$\mathbf{\longrightarrow}$};
\end{tikzpicture}
\caption{Adding six vertices: $(3^{p} \, | \, 3^{p}) \to (3^{p+3} \, | \, 3^{p+3})$.}
\label{fig:expansion3}
\end{figure}
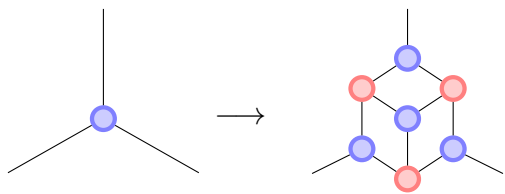

The pair $P=(3^{4r} \, | \, 4^{3r}), \; r \ge 2$ is always planar. To see that, we consider two concentric circles with the center at some point $O$ and the domain $U$ in the intersection of the annulus between the two circles and the central angle of size $\frac{2\pi}{r}$. Inscribe the drawing in Figure~\ref{fig:sector4} in the domain $U$ and then take the union of the images of this drawing under the rotations by $\frac{2\pi}{r} i, \; i=1, \dots, r$, about $O$. We obtain a plane drawing of the pair $P=(3^{4r} \, | \, 4^{3r}), \; r \ge 2$.
\begin{figure}[h!]
\pgfmathsetmacro{\ra}{4}
\pgfmathsetmacro{\rb}{5}
\pgfmathsetmacro{\rc}{6}
\pgfmathsetmacro{\co}{cos(60)}
\pgfmathsetmacro{\si}{sin(60)}
\pgfmathsetmacro{\coa}{cos(75)}
\pgfmathsetmacro{\sia}{sin(75)}
\pgfmathsetmacro{\coc}{cos(97.5)}
\pgfmathsetmacro{\sic}{sin(97.5)}
\tikzstyle{blue} = [circle,draw=blue!50,fill=blue!20,ultra thick,inner sep=0pt,minimum size=3mm]
\tikzstyle{red} = [circle,draw=red!50,fill=red!20,ultra thick,inner sep=0pt,minimum size=3mm]
\begin{tikzpicture}
\draw (\ra*\co,\ra*\si) arc (60:120:\ra);
\draw (\rb*\co,\rb*\si) arc (60:90:\rb); \draw (-\rb*\co,\rb*\si) arc (120:105:\rb);
\draw (\rc*\co,\rc*\si) arc (60:120:\rc);
\node [blue] (v1) at (\rc*\coa,\rc*\sia) {};
\node [blue] (v3) at (\ra*\coa,\ra*\sia) {};
\node [red] (v2) at (\rb*\coa,\rb*\sia) {};
\draw (v1)--(v2)--(v3);
\node [blue] (v4) at (-\rb*\coa,\rb*\sia) {};
\node [blue] (v5) at (0,\rb) {};
\node [red] (v6) at (\rc*\coc,\rc*\sic) {};
\node [red] (v7) at (\ra*\coc,\ra*\sic) {};
\draw (v4)--(v6)--(v5)--(v7)--(v4);
\end{tikzpicture}
\caption{The drawing in the domain $U$ for the pair $P=(3^{4r} \, | \, 4^{3r}), \; r \ge 2$.}
\label{fig:sector4}
\end{figure}
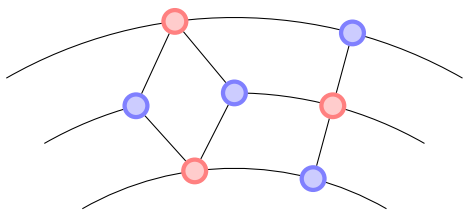
Similarly, for the pair $P=(3^{5r} \, | \, 5^{3r})$, where $r \ge 4$ is even, we start with two concentric circles with the center at $O$ and consider the domain $U$ in the intersection of the annulus between the two circles and the central angle of size $\frac{\pi}{r}$. We then inscribe the drawing in Figure~\ref{fig:sector5} in the domain $U$ and take the union of its images under the rotations by $\frac{\pi}{r} i, \; i=1, \dots, \frac12r$, about $O$. This gives a plane drawing of the pair $P=(3^{5r} \, | \, 5^{3r})$, for any even $r \ge 4$.

\begin{figure}[h!]
\pgfmathsetmacro{\ra}{4}
\pgfmathsetmacro{\rb}{5}
\pgfmathsetmacro{\rc}{6}
\pgfmathsetmacro{\rd}{7}
\pgfmathsetmacro{\re}{8}
\pgfmathsetmacro{\co}{cos(60)}
\pgfmathsetmacro{\si}{sin(60)}
\pgfmathsetmacro{\coaa}{cos(70)}
\pgfmathsetmacro{\siaa}{sin(70)}
\pgfmathsetmacro{\coa}{cos(75)}
\pgfmathsetmacro{\sia}{sin(75)}
\pgfmathsetmacro{\coab}{cos(80)}
\pgfmathsetmacro{\siab}{sin(80)}
\pgfmathsetmacro{\coc}{cos(85)}
\pgfmathsetmacro{\sic}{sin(85)}
\tikzstyle{blue} = [circle,draw=blue!50,fill=blue!20,ultra thick,inner sep=0pt,minimum size=3mm]
\tikzstyle{red} = [circle,draw=red!50,fill=red!20,ultra thick,inner sep=0pt,minimum size=3mm]
\begin{tikzpicture}
\draw (\ra*\co,\ra*\si) arc (60:120:\ra);
\draw (\rb*\co,\rb*\si) arc (60:70:\rb); \draw (0,\rb) arc (90:120:\rb);
\draw (\rc*\co,\rc*\si) arc (60:75:\rc); \draw (\rc*\coc,\rc*\sic) arc (85:100:\rc); \draw (-\rc*\coaa,\rc*\siaa) arc (110:120:\rc);
\draw (\rd*\co,\rd*\si) arc (60:70:\rd); \draw (0,\rd) arc (90:120:\rd);
\draw (\re*\co,\re*\si) arc (60:120:\re);
\node [blue] (v1) at (-\ra*\coa,\ra*\sia) {};
\node [red]  (v2) at (-\rb*\coa,\rb*\sia) {};
\node [blue] (v3) at (-\rc*\coaa,\rc*\siaa) {};
\node [blue] (v4) at (-\rc*\coab,\rc*\siab) {};
\node [red] (v5) at (-\rd*\coa,\rd*\sia) {};
\node [blue] (v6) at (-\re*\coa,\re*\sia) {};
\draw (v1)--(v2)--(v3)--(v5)--(v6); \draw (v2)--(v4)--(v5);
\node [blue] (v7) at (0,\rb) {};
\node [blue] (v8) at (0,\rd) {};
\node [blue] (v9) at (\rb*\coaa,\rb*\siaa) {};
\node [blue] (v10) at (\rb*\coab,\rb*\siab) {};
\node [blue] (v11) at (\rd*\coaa,\rd*\siaa) {};
\node [blue] (v12) at (\rd*\coab,\rd*\siab) {};
\node [red] (v13) at (\ra*\coab,\ra*\siab) {};
\node [red] (v14) at (\re*\coab,\re*\siab) {};
\node [red] (v15) at (\rc*\coc,\rc*\sic) {};
\node [red] (v16) at (\rc*\coa,\rc*\sia) {};
\draw (v14)--(v11)--(v16)--(v9)--(v13)--(v7)--(v15)--(v8)--(v14)--(v12)--(v16)--(v10)--(v15)--(v12); \draw (v10)--(v13);
\end{tikzpicture}
\caption{The drawing in the domain $U$ for the pair $P=(3^{5r} \, | \, 5^{3r}), \; r \ge 4$ even.}
\label{fig:sector5}
\end{figure}
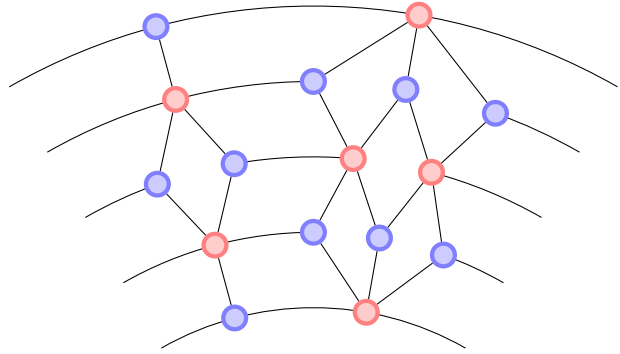

Plane drawings of the pairs $P=(3^{5r} \, | \, 5^{3r})$ for $r = 7$ and $r=9$ are given in Figures~\ref{a3b5r7} and \ref{fig:345527} respectively. Now for any odd $r \ge 11$, a plane drawing of the pair $P=(3^{5r} \, | \, 5^{3r})$ can be obtained by taking the disjoint union of the drawings for $r=7$ and for the corresponding even $r \ge 4$ (the resulting graph can be then made connected by the procedure shown in Figure~\ref{fig:connect}) .
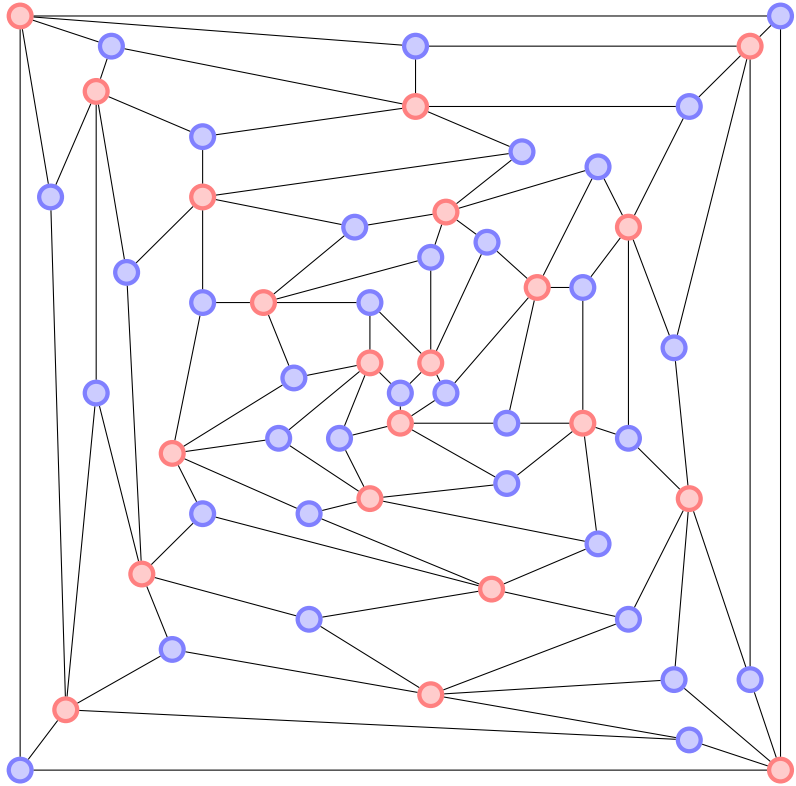
\begin{figure}[h!]
\tikzstyle{blue} = [circle,draw=blue!50,fill=blue!20,ultra thick,inner sep=0pt,minimum size=3mm]
\tikzstyle{red} = [circle,draw=red!50,fill=red!20,ultra thick,inner sep=0pt,minimum size=3mm]
\begin{tikzpicture}[scale=0.4]
\node [blue] (v7) at (-2.5,-5) {};
\node [blue] (v1) at (-3.5,-2) {};
\node [blue] (v3) at (-1,-5) {};
\node [blue] (v5) at (-4.5,-6.5) {};
\node [red] (v2) at (-1.5,-4) {};
\node [red] (v6) at (-3.5,-4) {};
\node [red] (v4) at (-2.5,-6) {};
\draw  (v1) edge (v2);
\draw  (v2) edge (v3);
\draw  (v3) edge (v4);
\draw  (v4) edge (v5);
\draw  (v5) edge (v6);
\draw  (v6) edge (v1);
\draw  (v2) edge (v7);
\draw  (v7) edge (v6);
\draw  (v4) edge (v7);

\node [blue] (v8) at (-1.5,-0.5) {};
\node [blue] (v9) at (0.35,0) {};

\node [blue] (v12) at (1,-8) {};
\node [blue] (v13) at (1,-6) {};
\draw  (v8) edge (v2);
\draw  (v2) edge (v9);

\draw  (v4) edge (v12);
\draw  (v4) edge (v13);
\node [red] (v14) at (-7,-2) {};
\node [red] (v15) at (2,-1.5) {};
\node [red] (v16) at (-3.5,-8.5) {};
\draw  (v1) edge (v14);
\draw  (v3) edge (v15);
\draw  (v5) edge (v16);

\node [blue] (v21) at (-5.5,-9) {};

\node [blue] (v19) at (4,2.5) {};
\node [blue] (v20) at (3.5,-1.5) {};

\draw  (v15) edge (v19);
\draw  (v15) edge (v20);
\draw  (v16) edge (v21);

\draw  (v8) edge (v14);
\draw  (v9) edge (v15);
\draw  (v15) edge (v13);
\draw  (v12) edge (v16);

\node [red] (v24) at (3.5,-6) {};

\node [red] (v25) at (5,0.5) {};

\draw  (v24) edge (v12);
\draw  (v13) edge (v24);
\draw  (v19) edge (v25);
\draw  (v25) edge (v20);

\node [red] (v27) at (-10,-7) {};
\draw  (v27) edge (v21);

\node [red] (v28) at (-1,1) {};
\draw  (v8) edge (v28);
\draw  (v28) edge (v9);
\draw  (v28) edge (v19);
\draw  (v20) edge (v24);
\node [blue] (v33) at (4,-10) {};
\node [red] (v32) at (0.5,-11.5) {};
\node  [blue] (v29) at (5,-6.5) {};
\node  [blue] (v31) at (5,-12.5) {};
\node [red] (v30) at (7,-8.5) {};
\draw  (v29) edge (v25);
\draw  (v24) edge (v29);
\draw  (v29) edge (v30);
\draw  (v30) edge (v31);
\draw  (v31) edge (v32);
\draw  (v33) edge (v24);
\draw  (v33) edge (v32);
\node [blue] (v11) at (-6.5,-6.5) {};
\node [blue] (v10) at (-6,-4.5) {};
\draw  (v10) edge (v6);
\draw  (v11) edge (v6);
\draw  (v33) edge (v16);
\draw  (v16) edge (v11);
\draw  (v14) edge (v10);
\draw  (v10) edge (v27);
\draw  (v11) edge (v27);
\node [blue] (v17) at (-9,-2) {};
\draw  (v17) edge (v14);
\draw  (v17) edge (v27);
\node [blue] (v18) at (-4,0.5) {};
\draw  (v18) edge (v28);
\draw  (v18) edge (v14);

\node [red] (v22) at (-9,1.5) {};
\node [red] (v26) at (-11,-11) {};
\node [blue] (v23) at (-11.5,-1) {};
\node [blue] (v34) at (-9,-9) {};
\draw  (v22) edge (v17);
\draw  (v23) edge (v22);
\draw  (v23) edge (v26);
\draw  (v26) edge (v34);
\draw  (v34) edge (v27);

\draw  (v22) edge (v18);
\node [blue] (v35) at (1.5,3) {};
\node [blue] (v37) at (7,4.5) {};
\node [red] (v36) at (-2,4.5) {};
\draw  (v35) edge (v28);
\draw  (v35) edge (v36);
\draw  (v36) edge (v37);
\draw  (v37) edge (v25);
\draw  (v35) edge (v22);
\node [blue] (v38) at (-9,3.5) {};
\node [red] (v39) at (-12.5,5) {};
\draw  (v38) edge (v36);
\draw  (v38) edge (v22);
\draw  (v38) edge (v39);
\draw  (v39) edge (v23);
\node [blue] (v40) at (-12.5,-5) {};
\draw  (v40) edge (v26);
\draw  (v40) edge (v39);
\node [blue] (v41) at (-12,6.5) {};
\draw  (v39) edge (v41);
\draw  (v41) edge (v36);
\node [red] (v44) at (-15,7.5) {};
\node [blue] (v42) at (-14,1.5) {};
\node [blue] (v45) at (-2,6.5) {};
\node [red] (v43) at (-13.5,-15.5) {};
\draw  (v42) edge (v43);
\draw  (v43) edge (v40);
\draw  (v42) edge (v39);
\draw  (v44) edge (v42);
\draw  (v44) edge (v45);
\draw  (v45) edge (v36);
\draw  (v44) edge (v41);
\node [blue] (v46) at (-10,-13.5) {};
\draw  (v46) edge (v26);
\draw  (v43) edge (v46);
\node [blue] (v47) at (-5.5,-12.5) {};
\draw  (v47) edge (v26);
\draw  (v21) edge (v32);
\draw  (v34) edge (v32);
\node [red] (v48) at (-1.5,-15) {};
\draw  (v47) edge (v48);
\draw  (v46) edge (v48);

\node [blue] (v49) at (6.5,-14.5) {};
\draw  (v49) edge (v48);
\draw  (v48) edge (v31);
\draw  (v30) edge (v49);
\node [blue] (v51) at (-15,-17.5) {};
\node [blue] (v50) at (7,-16.5) {};
\draw  (v50) edge (v43);
\draw  (v43) edge (v51);
\draw  (v50) edge (v48);
\node [red] (v52) at (10,-17.5) {};
\draw  (v51) edge (v52);
\draw  (v52) edge (v49);
\node [blue] (v53) at (9,-14.5) {};
\draw  (v52) edge (v53);
\draw  (v53) edge (v30);

\draw  (v50) edge (v52);
\node [red] (v54) at (9,6.5) {};
\draw  (v54) edge (v37);

\draw  (v53) edge (v54);
\draw  (v54) edge (v45);
\node [blue] (v55) at (6.5,-3.5) {};
\draw  (v55) edge (v30);
\draw  (v25) edge (v55);
\draw  (v55) edge (v54);
\node [blue] (v56) at (10,7.5) {};
\draw  (v54) edge (v56);
\draw  (v56) edge (v52);
\draw  (v44) edge (v56);
\draw  (v44) edge (v51);
\draw  (v47) edge (v32);
\end{tikzpicture}
\caption{Plane drawing of the pair $(3^{5r} \, | \, 5^{3r}), \; r=7$.}
\label{a3b5r7}
\end{figure}
%
%\newpage
%
\begin{figure}[h!]
\pgfmathsetmacro{\ra}{1}
\pgfmathsetmacro{\co}{\ra*cos(30)}
\pgfmathsetmacro{\si}{\ra*sin(30)}
\tikzstyle{blue} = [circle,draw=blue!50,fill=blue!20,ultra thick,inner sep=0pt,minimum size=3mm]
\tikzstyle{red} = [circle,draw=red!50,fill=red!20,ultra thick,inner sep=0pt,minimum size=3mm]
\begin{tikzpicture}
\node [blue] (v0) at (0,0) {}; %$v_0$
\node [blue] (v1) at (0,1) {}; %$v_1$
\node [blue] (v2) at (-\co,-\si) {}; %$v_2$
\node [blue] (v3) at (\co,-\si) {}; %$v_3$
\node [red] (u1) at (\co,\si) {}; %$u_1$
\node [red] (u2) at (-\co,\si) {}; %$u_2$
\node [red] (u3) at (0,-1) {}; %$u_3$
    \draw (v0)--(u1)--(v1)--(u2)--(v2)--(u3)--(v3)--(u1); \draw (u2)--(v0)--(u3);
\node [blue] (v4) at (2*\co,2*\si) {}; %$v_4$
\node [blue] (v5) at (-2*\co,2*\si) {}; %$v_5$
\node [blue] (v6) at (0,-2) {}; %$v_6$
\node [red] (u4) at (0,2) {}; %$u_4$
\node [red] (u5) at (-2*\co,-2*\si) {}; %$u_5$
\node [red] (u6) at (2*\co,-2*\si) {}; %$u_6$
\node [red] (u7) at (3*\co,\si) {}; %$u_7$
\node [red] (u8) at ($(v4)+(0,1)$) {}; %$u_8$
\node [red] (u9) at ($(v5)+(0,1)$) {}; %$u_9$
\node [red] (u10) at (-3*\co,\si) {}; %{\footnotesize $u_{10}$}
\node [red] (u11) at ($(v2)+(0,-2)$) {}; %{\footnotesize $u_{11}$}
\node [red] (u12) at ($(v3)+(0,-2)$) {}; %{\footnotesize $u_{12}$}
\node [red] (u13) at (4*\co,4*\si) {}; %{\footnotesize $u_{13}$}
\node [blue] (v7) at (3*\co,-\si) {}; %$v_7$
\node [blue] (v8) at ($(u1)+(0,2)$) {}; %$v_8$
\node [blue] (v9) at ($(u2)+(0,2)$) {}; %$v_9$
\node [blue] (v10) at (-3*\co,-\si) {}; %{\footnotesize $v_{10}$}
\node [blue] (v11) at ($(u5)+(0,-1)$) {}; %{\footnotesize $v_{11}$}
\node [blue] (v12) at ($(u6)+(0,-1)$) {}; %{\footnotesize $v_{12}$}
\node [blue] (v13) at ($(v4)+(0,-1)$) {}; %{\footnotesize $v_{13}$}
\node [blue] (v14) at ($(v9)+(0,-1)$) {}; %{\footnotesize $v_{14}$}
\node [blue] (v15) at ($(v2)+(0,-1)$) {}; %{\footnotesize $v_{15}$}
\draw (v2)--(u5)--(v10)--(u10)--(v5)--(u2); \draw (u1)--(v4)--(u8)--(v8)--(u4)--(v1); \draw (u3)--(v6)--(u12)--(v12)--(u6)--(v3);
\draw (v4)--(u7)--(v7)--(u6)--(v13)--(u1); \draw (u7)--(v13);
\draw (u2)--(v14)--(u4)--(v9)--(u9)--(v5); \draw (u9)--(v14);
\draw (v6)--(u11)--(v11)--(u5)--(v15)--(u11); \draw (u3)--(v15);
\node [red] (u14) at ($(v8)+(0,1)$) {}; %{\footnotesize $u_{14}$}
\node [red] (u15) at ($(v9)+(0,1)$) {}; %{\footnotesize $u_{15}$}
\node [red] (u16) at (-4*\co,4*\si) {}; %{\footnotesize $u_{16}$}
\node [red] (u17) at ($(u16)+(0,-3)$) {}; %{\footnotesize $u_{17}$}
\node [red] (u18) at ($(u10)+(0,-3)$) {}; %{\footnotesize $u_{18}$}
\node [red] (u19) at (0,-4) {}; %{\footnotesize $u_{19}$}
\node [red] (u20) at ($(u7)+(0,-3)$) {}; %{\footnotesize $u_{20}$}
\node [red] (u21) at ($(u13)+(0,-3)$) {}; %{\footnotesize $u_{21}$}
\node [blue] (v16) at (0,4) {}; %{\footnotesize $v_{16}$}
\node [blue] (v17) at ($(u10)+(0,2)$) {}; %{\footnotesize $v_{17}$}
\node [blue] (v18) at ($(u10)+(-\co,\si)$) {}; %{\footnotesize $v_{18}$}
\node [blue] (v19) at (-4*\co,-4*\si) {}; %{\footnotesize $v_{19}$}
\node [blue] (v20) at ($(u11)+(0,-1)$) {}; %{\footnotesize $v_{20}$}
\node [blue] (v21) at ($(u12)+(0,-1)$) {}; %{\footnotesize $v_{21}$}
\node [blue] (v22) at (4*\co,-4*\si) {}; %{\footnotesize $v_{22}$}
\node [blue] (v23) at ($(u13)+(0,-1)$) {}; %{\footnotesize $v_{23}$}
\node [blue] (v24) at (3*\co,3*\si) {}; %{\footnotesize $v_{24}$}
\node [blue] (v25) at ($(v24)+(0,1)$) {}; %{\footnotesize $v_{25}$}
\node [blue] (v26) at (0,3) {}; %{\footnotesize $v_{26}$}
\node [blue] (v27) at (-3*\co,3*\si) {}; %{\footnotesize $v_{27}$}
\node [blue] (v28) at (-3*\co,-3*\si) {}; %{\footnotesize $v_{28}$}
\node [blue] (v29) at (0,-3) {}; %{\footnotesize $v_{29}$}
\node [blue] (v30) at (3*\co,-3*\si) {}; %{\footnotesize $v_{30}$}
\draw (v24)--(u7)--(v23)--(u13)--(v25)--(u8)--(v24)--(u13);
\draw (v8)--(u14)--(v16)--(u15)--(v26)--(u4); \draw (u14)--(v26); \draw (u15)--(v9);
\draw (v27)--(u9)--(v17)--(u16)--(v18)--(u10)--(v27)--(u16);
\draw (v10)--(u17)--(v19)--(u18)--(v28)--(u5); \draw (u17)--(v28); \draw (v11)--(u18);
\draw (v29)--(u11)--(v20)--(u19)--(v21)--(u12)--(v29)--(u19);
\draw (v12)--(u20)--(v22)--(u21)--(v30)--(u6); \draw (u20)--(v30); \draw (v7)--(u21);
\draw (v20)--(u18); \draw (u21)--(v23); \draw (u15)--(v17);
\node [blue] (v31) at ($(u8)+(0,1)$) {}; %{\footnotesize $v_{31}$}
\node [blue] (v32) at ($(u17)+(0,1)$) {}; %{\footnotesize $v_{32}$}
\node [blue] (v33) at ($(u6)+(0,-2)$) {}; %{\footnotesize $v_{33}$}
\draw (v23)--(u21); \draw (v20)--(u18); \draw (v17)--(u15);
\draw (u8)--(v31)--(u14); \draw (u10)--(v32)--(u17); \draw (u12)--(v33)--(u20);
\node [red] (u22) at ($(v25)+(0,1)$) {}; %{\footnotesize $u_{22}$}
\node [red] (u23) at ($(v18)+(-\co,-\si)$) {}; %{\footnotesize $u_{23}$}
\node [red] (u24) at ($(v33)+(0,-1)$) {}; %{\footnotesize $u_{24}$}
\draw (v25)--(u22)--(v31); \draw (v18)--(u23)--(v32); \draw (v21)--(u24)--(v33);
\node [blue] (v34) at ($(u21)+(\co,\si)$) {}; %{\footnotesize $v_{34}$}
\node [blue] (v35) at ($(u15)+(-\co,\si)$) {}; %{\footnotesize $v_{35}$}
\node [blue] (v36) at ($(u18)+(0,-1)$) {}; %{\footnotesize $v_{36}$}
\draw (u13)--(v34)--(u21); \draw (u15)--(v35)--(u16); \draw (u18)--(v36)--(u19);
\node [red] (u25) at ($(u21)+(2*\co,-2*\si)$) {}; %{\footnotesize $u_{25}$}
\node [red] (u26) at ($(u15)+(0,2)$) {}; %{\footnotesize $u_{26}$}
\node [red] (u27) at ($(u18)+(-2*\co,-2*\si)$) {}; %{\footnotesize $u_{27}$}
\draw (v22)--(u25)--(v34); \draw (v16)--(u26)--(v35); \draw (v19)--(u27)--(v36);
\node [blue] (v37) at ($(v11)+(0,-2.7)$) {}; %{\footnotesize $v_{37}$}
\node [blue] (v38) at ($(v34)+(0,-2)$) {}; %{\footnotesize $v_{38}$}
\node [blue] (v39) at ($(v7)+(3*\co,3*\si)$) {}; %{\footnotesize $v_{39}$}
\node [blue] (v40) at ($(v16)+(0,1)$) {}; %{\footnotesize $v_{40}$}
\node [blue] (v41) at ($(v9)+(-3*\co,3*\si)$) {}; %{\footnotesize $v_{41}$}
\node [blue] (v42) at ($(u23)+(0,-3)$) {}; %{\footnotesize $v_{42}$}
\draw (u19)--(v37)--(u24)--(v38); \draw (v37)--(u27)--(v42)--(u23); \draw (u20)--(v38)--(u25)--(v39)--(u13); \draw (v39)--(u22)--(v40)--(u14);
\draw (v40)--(u26)--(v41)--(u16); \draw (v41)--(u23)--(v42)--(u17);
\node [blue] (v43) at ($(u25)+(0,-1)$) {}; %{\footnotesize $v_{43}$}
\node [blue] (v44) at ($(u23)+(-\co,-\si)$) {}; %{\footnotesize $v_{44}$}
\draw (u24)--(v43)--(u25); \draw (v43) to [out=60,in=0] (u22);
\draw (v44)--(u23); \draw (v44) to [out=90,in=180] (u26); \draw (v44) to [out=-90,in=120] (u27);
\end{tikzpicture}
\caption{Plane drawing of the pair $(3^{5r} \, | \, 5^{3r}), \, r = 9$.}
\label{fig:345527}
\end{figure}
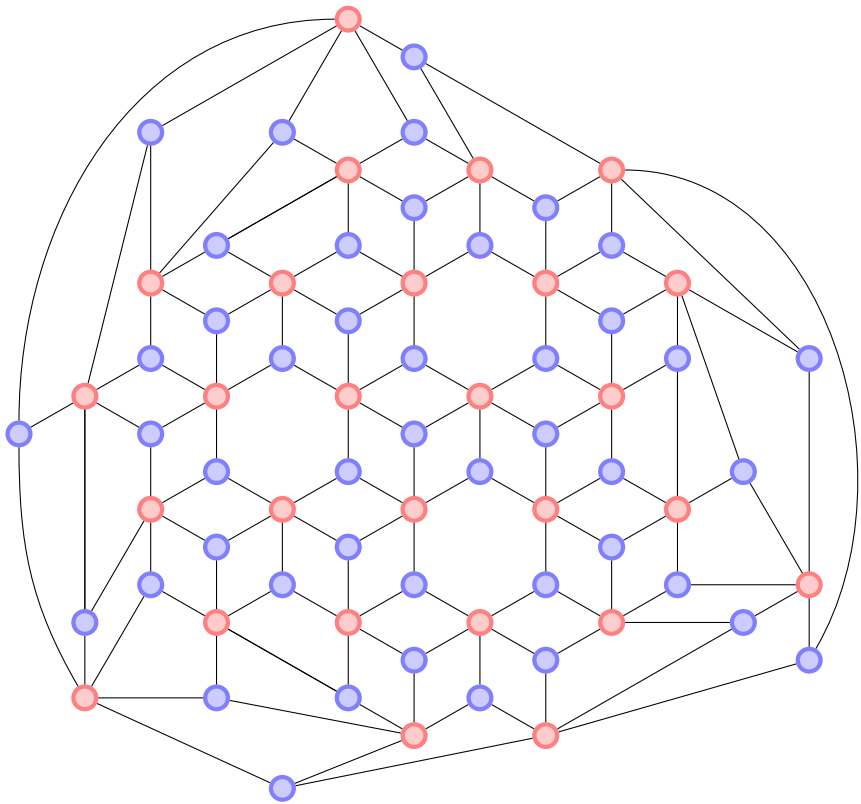

\underline{Non-planar cases.} To complete the proof we have to show that the pairs $(3^{5} \, | \, 3^{5})$ and $(3^{25} \, | \, 5^{15})$ are not planar. In the both cases, we follow the same strategy, although the second case is substantially more technical.

%Then its plane drawing contains one hexagonal face and six quadrilateral faces.
Suppose the pair $(3^{5} \, | \, 3^{5})$ is planar. There are obviously $15$ edges and $10$ vertices, so using Euler's formula, it is not difficult to see that there is necessarily one hexagonal face and six quadrilateral faces. Take a point in the hexagonal face and join it to the three vertices belonging to the second part by non-crossing arcs. We obtain a plane drawing of the pair $(3^{6} \, | \, 4^3, 3^2)$ in which all the faces are quadrilaterals. We now join by a diagonal, in every quadrilateral, the two vertices belonging to the second part, and then remove from the resulting drawing the vertices of the first part and the interiors of all the edges incident to those vertices. We obtain a plane drawing of the sequence $(4^3, 3^2)$. Note that by construction, the faces of that drawing are in one-to-one correspondence with the vertices of the first part and moreover, as all those vertices were of degree $3$, all the faces of the drawing are triangles. We also note that the three vertices of degree $4$ lie on the boundary of one face.

If the pair $(3^{25} \, | \, 5^{15})$ is planar, then its plane drawing has $75$ edges and $40$ vertices. Then there are $37$ faces and so, similar to the previous case, the drawing has one hexagonal face and all the other faces are quadrilateral. Then the same procedure applied to the pair $(3^{25} \, | \, 5^{15})$ produces a plane drawing of the sequence $(6^3, 5^{12})$ all of whose faces are triangles and such that the three vertices of degree $6$ lie on the boundary of one face.

\begin{lemma} \label{l:simple}
The underlying graphs of the both drawings are simple.
\end{lemma}
\begin{proof}
  We start by considering both drawings simultaneously. The fact that there are no loops follows from the fact that the original drawings had no multiple edges.

  Suppose a pair of vertices $A, B$ is joined by two edges $e_1, e_2$. We can view our drawing as a drawing on the sphere. Consider the two domains $U_1, U_2$ into which the simple closed curve $c=e_1 \cup e_2$ splits the sphere. For $i=1,2$, denote $\delta_i(A)$ and $\delta_i(B)$ the number of edges incident to $A$ (respectively $B$) which point towards $U_i$. As all the faces of the drawing are triangular and there are no loops, we have $\delta_i(A), \delta_i(B) \ge 1$. Let us assume that $\delta_1(A) \le \delta_1(B)$. Suppose first that $\delta_1(A)=1$. If $\delta_1(B)=1$, we get a drawing in the closure of $U_1$ as on the left in Figure~\ref{fig:simple}. This is a contradiction, as both domains must be faces and the graph has no vertices of degree $2$. If $\delta_1(B)=2$, we get a drawing in the closure of $U_1$ as in the middle in Figure~\ref{fig:simple} (there may be some more edges incident to $C$ which are not shown). This is again a contradiction, as the two edges joining $B$ and $C$ must lie on the boundary of a single triangular face. It follows that $\delta_1(B) \ge 3$. Therefore we always have $\delta_1(A) + \delta_1(B) \ge 4$, and similarly $\delta_2(A) + \delta_2(B) \ge 4$. Then the sum of the degrees of the vertices $A$ and $B$ is at least $12$. This establishes the claim for the graph $(4^3, 3^2)$.
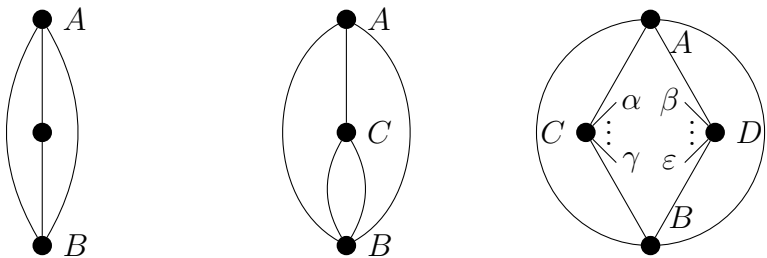
\begin{figure}[h!]
\pgfmathsetmacro{\rx}{4}
\tikzstyle{black} = [circle,draw=black,fill=black,ultra thick,inner sep=0pt,minimum size=2mm]
\begin{tikzpicture}
\node [black] (A1) at (0,3) [label=0:$A$] {};
\node [black] (B1) at (0,0) [label=0:$B$] {};
\node [black] (C1) at (0,1.5) {};
\draw (A1)--(B1);
\draw (B1) to [out=120,in=-120] (A1);
\draw (B1) to [out=60,in=-60] (A1);
\node [black] (A2) at (\rx,3) [label=0:$A$] {};
\node [black] (B2) at (\rx,0) [label=0:$B$] {};
\node [black] (C2) at (\rx,1.5) [label=0:$C$] {};
\draw (A2)--(C2);
\draw (B2) to [out=150,in=-150] (A2);
\draw (B2) to [out=30,in=-30] (A2);
\draw (B2) to [out=60,in=-60] (C2);
\draw (B2) to [out=120,in=-120] (C2);
\node [black] (A3) at (2*\rx,3) [label={[xshift=0.4cm, yshift=-0.72cm]:$A$}] {};
\node [black] (B3) at (2*\rx,0) [label=45:$B$] {};
\node [black] (C3) at (2*\rx-0.85,1.5) [label=180:$C$] {};
\node [black] (D3) at (2*\rx+0.85,1.5) [label=0:$D$] {};
\draw (A3)--(C3)--(B3)--(D3)--(A3);
\draw (2*\rx,1.5) circle (1.5);
\draw ($(C3)+(0.4,0.4)$)--(C3)--($(C3)+(0.4,-0.4)$);
\draw ($(D3)+(-0.4,0.4)$)--(D3)--($(D3)+(-0.4,-0.4)$);
\node at ($(C3)+(0.3,0.1)$) {$\vdots$};
\node at ($(D3)+(-0.3,0.1)$) {$\vdots$};
\node at ($(C3)+(0.6,0.4)$) {$\alpha$};
\node at ($(C3)+(0.6,-0.4)$) {$\gamma$};
\node at ($(D3)+(-0.6,0.4)$) {$\beta$};
\node at ($(D3)+(-0.6,-0.4)$) {$\varepsilon$};
\end{tikzpicture}
\caption{The part of the drawing in the closure of $U_1$.}
\label{fig:simple}
\end{figure}
  For the graph $(6^3, 5^{12})$, the only possibility is that both $A$ and $B$ are of degree $6$. Continuing with the drawing of the graph $(6^3, 5^{12})$ we notice that no two vertices can be joined by more than two edges (otherwise, by the argument above, their degrees are too high). For every pair of vertices joined by two edges, the union of the edges is a simple closed curve $c$; call it a \emph{circle}. No two circles may cross, but some of them can touch at a vertex (note that we have only three vertices of degree $6$, so there are no more than three circles). We can now choose a circle $c$ with the property that one of the domains $U_1$ bounded by it contains no points of any other circle (call such a domain \emph{minimal}). Denote $U_2$ the other domain bounded by $c$ and let $A, B \in c$ be the vertices of the drawing joined by two edges $e_1, e_2$ with $c=e_1 \cup e_2$. With the notation as above, we cannot have $\delta_1(A)=1$ or $\delta_1(B)=1$ as this creates double edges lying in the closure of $U_1$ hence contradicting the fact that $U_1$ is minimal. Then $\delta_1(A), \delta_1(B) \ge 2$, and so $\delta_2(A), \delta_2(B) \le 2$. From the argument above, $\delta_2(A) \ge 1$ and if $\delta_2(A)=1$, then $\delta_2(B) \ge 3$. It follows that $\delta_i(A)=\delta_i(B) = 2$, for $i=1,2$. Now consider the part of the drawing lying in the closure of $U_1$. The two edges incident to $A$ and pointing towards $U_1$ cannot have a common endpoint (other than $A$). The same is true for the two edges incident to $B$ and pointing towards $U_1$, so the two triangular domains containing $e_1$ (respectively $e_2$) on their boundaries are as on the right in Figure~\ref{fig:simple}. The degree of each of the vertices $C, D$ is at least $5$; the starting segments of some of the edges incident to them are shown in Figure~\ref{fig:simple}. In the rotation diagram at $C$, the segments $C\gamma$ (respectively $C\alpha$) is the first (respectively the last) when we turn in the positive direction from $CB$ to $CA$. For $D$, we label the starting segments similarly. Then the path $\alpha CAD\beta$ lies on the boundary of a face, so the segments $C\alpha$ and $D\beta$ lie on the same edge. By a similar argument, the segments $C\gamma$ and $D\varepsilon$ lie on the same edge. This gives two edges joining $C$ and $D$ which contradicts the fact that $U_1$ is minimal.
\end{proof}

We need to show that the degree sequence $(4^3, 3^2)$ (respectively $(6^3,5^{12})$) admits no plane drawing all of whose faces are triangles, with all the three vertices of degree $4$ (respectively $6$) lying on the boundary of one face and such that the underlying graph is simple. Assuming such a drawing exists, we call a \emph{partial drawing} a drawing of a subset of vertices with some edges connecting them, and with starting segments of all the other edges adjacent to these vertices drawn in such a way that at each vertex, we have the same rotation diagram as in the final drawing.

For the degree sequence $(4^3, 3^2)$, we start with the partial drawing containing three vertices of degree $4$, the edges connecting them and the starting segments of all the other edges incident to these vertices, as in Figure~\ref{fig:4332}. The triangular face having the vertices of degree $4$ on its boundary is the outer face. An easy inspection shows that we cannot add two vertices of degree $3$ and connect the edges to get the required drawing.
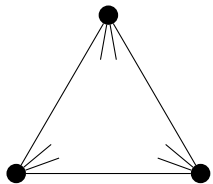
\begin{figure}[h!]
\pgfmathsetmacro{\ra}{1.4}
\pgfmathsetmacro{\rb}{0.6}
\pgfmathsetmacro{\co}{cos(30)}
\pgfmathsetmacro{\si}{sin(30)}
\pgfmathsetmacro{\coa}{cos(10)}
\pgfmathsetmacro{\sia}{sin(10)}
\tikzstyle{black} = [circle,draw=black,fill=black,ultra thick,inner sep=0pt,minimum size=2mm]
\begin{tikzpicture}
\node [black] (v1) at (0,\ra) {};
\node [black] (v2) at (-\ra*\co,-\ra*\si) {};
\node [black] (v3) at (\ra*\co,-\ra*\si) {};
\draw (v1)--(v2)--(v3)--(v1);
\draw ($(v2)+({\rb*cos(20)},{\rb*sin(20)})$)--(v2)--($(v2)+({\rb*cos(40)},{\rb*sin(40)})$);
\draw ($(v3)+({-\rb*cos(20)},{\rb*sin(20)})$)--(v3)--($(v3)+({-\rb*cos(40)},{\rb*sin(40)})$);
\draw ($(v1)+({\rb*\sia},{-\rb*\coa})$)--(v1)--($(v1)+({-\rb*\sia},{-\rb*\coa})$);
\end{tikzpicture}
\caption{The partial drawing for the sequence $(4^3, 3^2)$.}
\label{fig:4332}
\end{figure}

Consider the degree sequence $(6^3, 5^{12})$. Note that this sequence \emph{is} planar \cite{SH, GM}, but, as we will show, in no plane drawing there is a face with three vertices of degree $6$. We also note that the dual to a planar graph with such a degree sequence is known as a \emph{fullerene graph} \cite{Mal}; such graphs are extensively studied in chemistry \cite{FM}.

Assuming that the required drawing exists, we start with a partial drawing as on the left in Figure~\ref{fig:63512start}, with the three vertices of degree $6$ and with the corresponding face to be the outer domain of the circle. All the remaining vertices will be of degree $5$.
\begin{figure}[h!]
\pgfmathsetmacro{\ra}{2.5}
\pgfmathsetmacro{\rb}{1}
\pgfmathsetmacro{\rc}{0.7}
\pgfmathsetmacro{\rx}{8}
\pgfmathsetmacro{\co}{cos(30)}
\pgfmathsetmacro{\si}{sin(30)}
\pgfmathsetmacro{\coa}{cos(10)}
\pgfmathsetmacro{\sia}{sin(10)}
\pgfmathsetmacro{\cob}{cos(20)}
\pgfmathsetmacro{\sib}{sin(20)}
\tikzstyle{black} = [circle,draw=black,fill=black,ultra thick,inner sep=0pt,minimum size=2mm]
\begin{tikzpicture}
\node [black] (v1) at (0,\ra) {};
\node [black] (v2) at (-\ra*\co,-\ra*\si) {};
\node [black] (v3) at (\ra*\co,-\ra*\si) {};
\draw (0,0) circle (\ra);
\draw ($(v1)+({\rb*\si},{-\rb*\co})$)--(v1)--($(v1)+({-\rb*\si},{-\rb*\co})$);
\draw ($(v1)+({\rb*\sia},{-\rb*\coa})$)--(v1)--($(v1)+({-\rb*\sia},{-\rb*\coa})$);
\draw ($(v2)+({\rb*cos(40)},{\rb*sin(40)})$)--(v2)--($(v2)+({\rb*\cob},{\rb*\sib})$);
\draw ($(v2)+(\rb,0)$)--(v2)--($(v2)+({\rb*cos(60)},{\rb*sin(60)})$);
\draw ($(v3)+({-\rb*cos(40)},{\rb*sin(40)})$)--(v3)--($(v3)+({-\rb*\cob},{\rb*\sib})$);
\draw ($(v3)+(-\rb,0)$)--(v3)--($(v3)+({-\rb*cos(60)},{\rb*sin(60)})$);
\node [black] (w1) at (\rx,\ra) {};
\node [black] (w2) at (\rx-\ra*\co,-\ra*\si) {};
\node [black] (w3) at (\rx+\ra*\co,-\ra*\si) {};
\draw (\rx,0) circle (\ra);
\draw ($(w1)+({\rb*\sia},{-\rb*\coa})$)--(w1)--($(w1)+({-\rb*\sia},{-\rb*\coa})$);
\draw ($(w2)+({\rb*cos(40)},{\rb*sin(40)})$)--(w2)--($(w2)+({\rb*\cob},{\rb*\sib})$);
\draw ($(w3)+({-\rb*cos(40)},{\rb*sin(40)})$)--(w3)--($(w3)+({-\rb*\cob},{\rb*\sib})$);
\draw (w1)--(w2)--(w3)--(w1);
\node [black] (w6) at ($0.5*(w1)+0.5*(w2)$) {};
\node [black] (w4) at ($0.5*(w2)+0.5*(w3)$) {};
\node [black] (w5) at ($0.5*(w3)+0.5*(w1)$) {};
\draw ($(w4)+({-\rc*\cob},{\rc*\sib})$)--(w4)--($(w4)+({\rc*\cob},{\rc*\sib})$); \draw ($(w4)+(0,\rc)$)--(w4);
\draw ($(w6)+({-\rc*\sia},{-\rc*\coa})$)--(w6)--($(w6)+({\rc*cos(40)},{\rc*sin(40)})$); \draw ($(w6)+(\rc*\co,-\rc*\si)$)--(w6);
\draw ($(w5)+({\rc*\sia},{-\rc*\coa})$)--(w5)--($(w5)+({-\rc*cos(40)},{\rc*sin(40)})$); \draw ($(w5)+(-\rc*\co,-\rc*\si)$)--(w5);
\end{tikzpicture}
\caption{The starting partial drawing for the sequence $(6^3, 5^{12})$.}
\label{fig:63512start}
\end{figure}
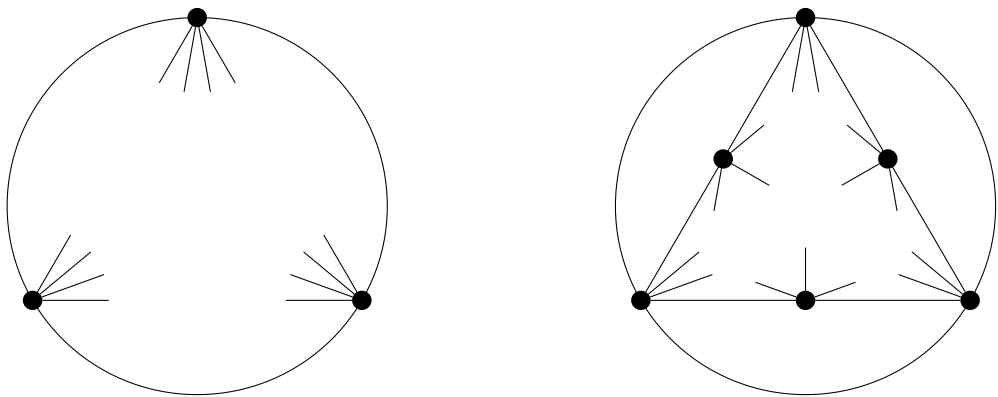
For each of the three edges, consider the two starting segments closest to it in the rotation diagrams at its endpoints. Then the edge and the two segments lie on the boundary of a single face, so that the edges of the drawing which extend these two starting segment share a common vertex. That vertex cannot be one of the three existing vertices and the three new vertices so constructed are pairwise distinct (as the graph is simple). We get the partial drawing as on the right in Figure~\ref{fig:63512start}. The starting segments at the three new vertices point towards the hexagonal domain, as the three triangular domains are faces. %This fact follows from Lemma~\ref{l:badt} below which we will use several times throughout the proof. each of

For $a, b, c \ge 0$, we call an \emph{$(a,b,c)$-triangle} a partial drawing consisting of a drawing of a $3$-cycle $ABC$ with $a$ (respectively $b, c$) starting segments incident to vertex $A$ (respectively $B, C$), with all the starting segments pointing to the same domain $U$ in the complement of the drawing of the $3$-cycle $ABC$. We say that a triangle is \emph{bad}, if there is no graph drawing in the closure of $U$ for which all the faces are triangular, all the vertices are either the vertices $A, B, C$ or lie in $U$ and have degree $5$, and the only edges incident to $A, B, C$ are the extensions of the given starting segments.

\begin{lemma} \label{l:badt}
  The triangles $(x,0,0), \, (x,y,0), \, (x, 1, 1), \, (2,2,1)$, where $x, y \in \mathbb{N}$, are bad.
\end{lemma}
\begin{proof} %(note that there can be just a single starting segment incident to $A$)
  For the triangle $(x,0,0), \, x >0$, consider the partial drawing on the left in Figure~\ref{fig:bad}. Then in any drawing extending it, the path $\alpha ABC \beta$ lies on the boundary of a single face which contradicts the fact that all the faces must be triangular. Consider the triangle $(x,y,0), \, x, y >0$, and assume that there are no starting segments attached to $C$ and that the vertices $A, B, C$ are labelled in the positive direction. Let $A\alpha$ be the closest starting segment to $AC$ in the rotation diagram at $A$, and $B\beta$ be the closest starting segment to $BC$ in the rotation diagram at $B$. Then $A\alpha$ and $B\beta$ must be parts of the same edge which contradicts the fact that the graph is simple.
\begin{figure}[h!]
\pgfmathsetmacro{\ra}{2.5}
\pgfmathsetmacro{\rb}{1}
\pgfmathsetmacro{\rc}{0.7}
\pgfmathsetmacro{\rx}{5.3}
\pgfmathsetmacro{\co}{cos(30)}
\pgfmathsetmacro{\si}{sin(30)}
\pgfmathsetmacro{\coa}{cos(10)}
\pgfmathsetmacro{\sia}{sin(10)}
\pgfmathsetmacro{\cob}{cos(20)}
\pgfmathsetmacro{\sib}{sin(20)}
\tikzstyle{black} = [circle,draw=black,fill=black,ultra thick,inner sep=0pt,minimum size=2mm]
\begin{tikzpicture}
\node [black] (vA) at (0,\ra) [label=0:$A$] {};
\node [black] (vB) at (-\ra*\co,-\ra*\si) [label=-90:$B$] {};
\node [black] (vC) at (\ra*\co,-\ra*\si) [label=-90:$C$] {};
\draw (vA)--(vB)--(vC)--(vA);
\node (p) at ($(vA)+(0,-1.3)$) {${\alpha \; \beta}$};
\draw ($(vA)+(-0.2,-\rb)$)--(vA)--($(vA)+(0.2,-\rb)$);
\draw ($(vA)+(0,-\rb)$)--(vA);
\node [black] (wA) at (\rx,\ra) [label=0:$A$] {};
\node [black] (wB) at (\rx-\ra*\co,-\ra*\si) [label=-90:$B$] {};
\node [black] (wC) at (\rx+\ra*\co,-\ra*\si) [label=-90:$C$] {};
\draw (wA)--(wB)--(wC)--(wA);
\node (p) at ($(wA)+(0,-1.3)$) {${\alpha \; \beta}$}; %{\footnotesize $1$} [label=-90:${1 \dots x}$]
\draw ($(wA)+(-0.2,-\rb)$)--(wA)--($(wA)+(0.2,-\rb)$);
\draw ($(wA)+(0,-\rb)$)--(wA);
\node (x) at ($(wB)+(1.2*\co,1.2*\si)$) {$\gamma$}; %[label=30:$\gamma$]
\node (y) at ($(wC)+(-1.2*\co,1.2*\si)$) {$\delta$}; %[label=150:$\delta$]
\draw (x)--(wB); \draw (y)--(wC);
\node [black] (zA) at (2*\rx,\ra) [label=0:$A$] {};
\node [black] (zB) at (2*\rx-\ra*\co,-\ra*\si) [label=-90:$B$] {};
\node [black] (zC) at (2*\rx+\ra*\co,-\ra*\si) [label=-90:$C$] {};
\draw (zA)--(zB)--(zC)--(zA);
\node (p) at ($(zA)+(0,-1.3)$) {${\alpha \; \beta}$}; %{\footnotesize $1$} [label=-90:${1 \dots x}$]
\draw ($(zA)+(-0.2,-\rb)$)--(zA)--($(zA)+(0.2,-\rb)$);
\draw ($(zA)+(0,-\rb)$)--(zA);
\node [black] (zD) at (2*\rx,0) [label=-90:$D$] {};
\draw (zB)--(zD)--(zC);
\node (q) at ($(zD)+(0,0.7)$) {${\mu \quad \nu}$};
\draw ($(zD)+(-0.3,0.5)$)--(zD)--($(zD)+(0.3,0.5)$);
\draw ($(zD)+(0,0.5)$)--(zD);
\end{tikzpicture}
\caption{Bad triangles.}
\label{fig:bad}
\end{figure}
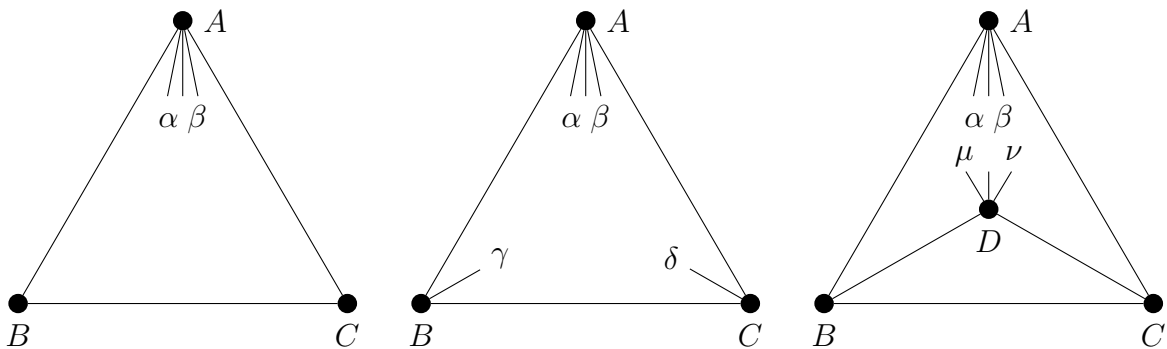
  Now consider the triangle $(x, 1, 1), \, x > 0$, as in the middle in Figure~\ref{fig:bad}. In any drawing which extends it, the path $\gamma BC\delta$ lies on the boundary of a triangular face, and so the edges containing the segments $B \gamma$ and $C \delta$ share a common vertex $D$. Note that the three starting segments at vertex $D$ point outside the triangle $DBC$ (otherwise $DBC$ is a bad triangle), as on the right in Figure~\ref{fig:bad}. But then the segments $D\mu$ and $A\alpha$ lie on the same edge and the segments $D\nu$ and $A\beta$ lie on the same edge which contradicts the fact that the graph is simple. By similar arguments, for the triangle $(2,2,1)$, we are forced to add two edges which then produces a bad triangle $(2,1,1)$.
\end{proof}

We now consider the six edges lying on the boundary of the inner triangular domain of the partial drawing on the right in Figure~\ref{fig:63512start}. Each of them lies on the boundary of a triangular face other than the one already shown in Figure~\ref{fig:63512start}. The third vertex of that face is the common endpoint of the two edges which extend the two closest starting segment at their endpoints of the edge. We first show that none of these vertices can be one of the existing six vertices. By symmetry, we can consider one of the six edges and then, as the graph is simple, if the two edges extending the two starting segments have one of the two existing vertices as their endpoint, then we get three possible partial drawings shown in Figure~\ref{fig:63512three}. In all three cases, we get a bad triangle, $(3,1,0),(2,2,1)$ and $(2,2,0)$ respectively.

%by Lemma~\ref{l:badt}. In the two cases in Figure~\ref{fig:63512two}, the bad triangles are. In the other two cases, we get bad triangles $(1,0,0)$ and .

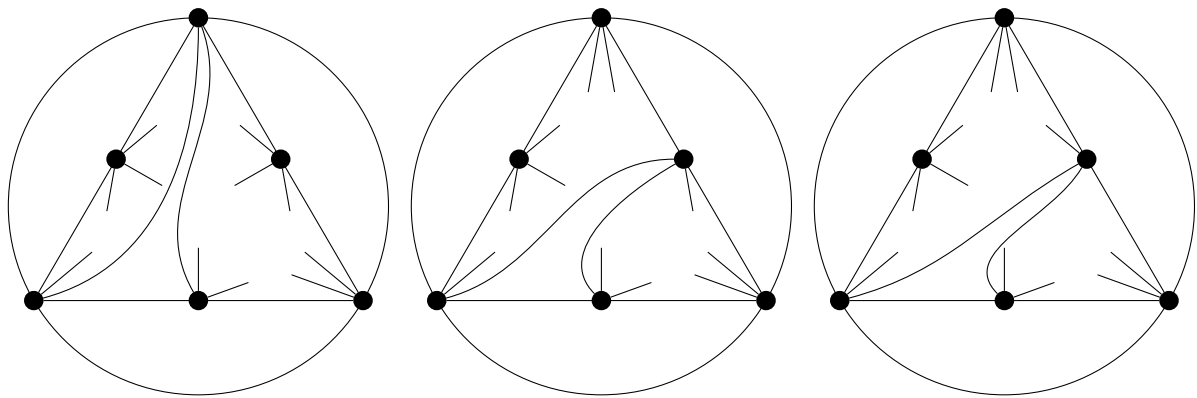
\begin{figure}[h!]
\pgfmathsetmacro{\ra}{2.5}
\pgfmathsetmacro{\rb}{1}
\pgfmathsetmacro{\rc}{0.7}
\pgfmathsetmacro{\rx}{5.3}
\pgfmathsetmacro{\co}{cos(30)}
\pgfmathsetmacro{\si}{sin(30)}
\pgfmathsetmacro{\coa}{cos(10)}
\pgfmathsetmacro{\sia}{sin(10)}
\pgfmathsetmacro{\cob}{cos(20)}
\pgfmathsetmacro{\sib}{sin(20)}
\tikzstyle{black} = [circle,draw=black,fill=black,ultra thick,inner sep=0pt,minimum size=2mm]
\begin{tikzpicture}
\node [black] (v1) at (0,\ra) {};
\node [black] (v2) at (-\ra*\co,-\ra*\si) {};
\node [black] (v3) at (\ra*\co,-\ra*\si) {};
\draw (0,0) circle (\ra);
\draw ($(v2)+({\rb*cos(40)},{\rb*sin(40)})$)--(v2);
\draw ($(v3)+({-\rb*cos(40)},{\rb*sin(40)})$)--(v3)--($(v3)+({-\rb*\cob},{\rb*\sib})$);
\draw (v1)--(v2)--(v3)--(v1);
\node [black] (v6) at ($0.5*(v1)+0.5*(v2)$) {};
\node [black] (v4) at ($0.5*(v2)+0.5*(v3)$) {};
\node [black] (v5) at ($0.5*(v3)+0.5*(v1)$) {};
\draw (v4)--($(v4)+({\rc*\cob},{\rc*\sib})$); \draw ($(v4)+(0,\rc)$)--(v4);
\draw ($(v6)+({-\rc*\sia},{-\rc*\coa})$)--(v6)--($(v6)+({\rc*cos(40)},{\rc*sin(40)})$); \draw ($(v6)+(\rc*\co,-\rc*\si)$)--(v6);
\draw ($(v5)+({\rc*\sia},{-\rc*\coa})$)--(v5)--($(v5)+({-\rc*cos(40)},{\rc*sin(40)})$); \draw ($(v5)+(-\rc*\co,-\rc*\si)$)--(v5);
\draw (v2) to [out=15,in=-90] (v1);
\draw (v4) to [out=120,in=-70] (v1);
\node [black] (w1) at (\rx,\ra) {};
\node [black] (w2) at (\rx-\ra*\co,-\ra*\si) {};
\node [black] (w3) at (\rx+\ra*\co,-\ra*\si) {};
\draw (\rx,0) circle (\ra);
\draw ($(w1)+({\rb*\sia},{-\rb*\coa})$)--(w1)--($(w1)+({-\rb*\sia},{-\rb*\coa})$);
\draw ($(w2)+({\rb*cos(40)},{\rb*sin(40)})$)--(w2);
\draw ($(w3)+({-\rb*cos(40)},{\rb*sin(40)})$)--(w3)--($(w3)+({-\rb*\cob},{\rb*\sib})$);
\draw (w1)--(w2)--(w3)--(w1);
\node [black] (w6) at ($0.5*(w1)+0.5*(w2)$) {};
\node [black] (w4) at ($0.5*(w2)+0.5*(w3)$) {};
\node [black] (w5) at ($0.5*(w3)+0.5*(w1)$) {};
\draw (w4)--($(w4)+({\rc*\cob},{\rc*\sib})$); \draw ($(w4)+(0,\rc)$)--(w4);
\draw ($(w6)+({-\rc*\sia},{-\rc*\coa})$)--(w6)--($(w6)+({\rc*cos(40)},{\rc*sin(40)})$); \draw ($(w6)+(\rc*\co,-\rc*\si)$)--(w6);
\draw ($(w5)+({\rc*\sia},{-\rc*\coa})$)--(w5);
\draw (w2) to [out=15,in=-180] (w5);
\draw (w4) to [out=135,in=-150] (w5);
\node [black] (w1) at (2*\rx,\ra) {};
\node [black] (w2) at (2*\rx-\ra*\co,-\ra*\si) {};
\node [black] (w3) at (2*\rx+\ra*\co,-\ra*\si) {};
\draw (2*\rx,0) circle (\ra);
\draw ($(w1)+({\rb*\sia},{-\rb*\coa})$)--(w1)--($(w1)+({-\rb*\sia},{-\rb*\coa})$);
\draw ($(w2)+({\rb*cos(40)},{\rb*sin(40)})$)--(w2);
\draw ($(w3)+({-\rb*cos(40)},{\rb*sin(40)})$)--(w3)--($(w3)+({-\rb*\cob},{\rb*\sib})$);
\draw (w1)--(w2)--(w3)--(w1);
\node [black] (w6) at ($0.5*(w1)+0.5*(w2)$) {};
\node [black] (w4) at ($0.5*(w2)+0.5*(w3)$) {};
\node [black] (w5) at ($0.5*(w3)+0.5*(w1)$) {};
\draw (w4)--($(w4)+({\rc*\cob},{\rc*\sib})$); \draw ($(w4)+(0,\rc)$)--(w4);
\draw ($(w6)+({-\rc*\sia},{-\rc*\coa})$)--(w6)--($(w6)+({\rc*cos(40)},{\rc*sin(40)})$); \draw ($(w6)+(\rc*\co,-\rc*\si)$)--(w6);
\draw ($(w5)+({-\rc*cos(40)},{\rc*sin(40)})$)--(w5);
\draw (w2) to [out=15,in=-150] (w5);
\draw (w4) to [out=135,in=-120] (w5);
\end{tikzpicture}
\caption{The third vertex of the triangular face cannot be one of the existing vertices.}
\label{fig:63512three}
\end{figure}

% (note that it cannot point towards any of the two triangular faces as this creates a bad triangle $(1,0,0)$)
We next show that all the six new vertices are pairwise distinct. Label the edges as in Figure~\ref{fig:63512notsame} and denote $A_i$ the third vertex of the corresponding triangular face. Clearly $A_1=A_2$ and $A_1=A_6$ are impossible. The cases when $A_1=A_5$ and $A_1=A_3$ can be obtained from one another by rotation. Hence it suffices to consider two cases: $A_1=A_4$ and $A_1=A_3$. If $A_1=A_4$, we have two possible partial drawings depending on the direction of the remaining starting segment at the vertex $A_1=A_4$; these drawings are shown in Figure~\ref{fig:63512notsame}. But in the drawing on the left, we are forced to have an edge shown in the broken line which then produces a non-simple graph. Similarly, in the drawing on the right, we must have an edge shown in the broken line which then produces a bad triangle $(2,1,1)$.

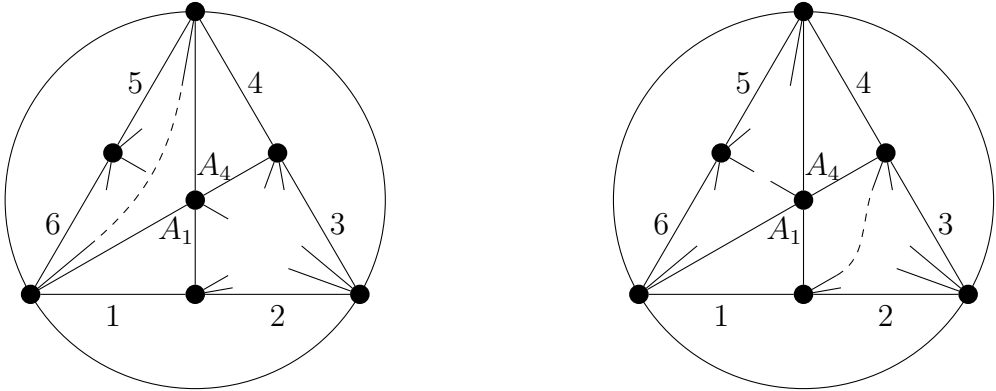
\begin{figure}[h!]
\pgfmathsetmacro{\ra}{2.5}
\pgfmathsetmacro{\rb}{1}
\pgfmathsetmacro{\rc}{0.7}
\pgfmathsetmacro{\rx}{8}
\pgfmathsetmacro{\co}{cos(30)}
\pgfmathsetmacro{\si}{sin(30)}
\pgfmathsetmacro{\coa}{cos(10)}
\pgfmathsetmacro{\sia}{sin(10)}
\pgfmathsetmacro{\cob}{cos(20)}
\pgfmathsetmacro{\sib}{sin(20)}
\tikzstyle{black} = [circle,draw=black,fill=black,ultra thick,inner sep=0pt,minimum size=2mm]
\begin{tikzpicture}
\node [black] (v1) at (0,\ra) {};
\node [black] (v2) at (-\ra*\co,-\ra*\si) {};
\node [black] (v3) at (\ra*\co,-\ra*\si) {};
\draw (0,0) circle (\ra);
\draw (v1)--($(v1)+({-\rb*\sia},{-\rb*\coa})$);
\draw ($(v2)+({\rb*cos(40)},{\rb*sin(40)})$)--(v2);
\draw ($(v3)+({-\rb*cos(40)},{\rb*sin(40)})$)--(v3)--($(v3)+({-\rb*\cob},{\rb*\sib})$);
\node [black] (v6) at ($0.5*(v1)+0.5*(v2)$) {};
\node [black] (v4) at ($0.5*(v2)+0.5*(v3)$) {};
\node [black] (v5) at ($0.5*(v3)+0.5*(v1)$) {};
\draw ($(v4)+({0.5*\rb*\coa},{0.5*\rb*\sia})$)--(v4)--($(v4)+({0.5*\rb*\co},{0.5*\rb*\si})$);
\draw ($(v6)+({-0.5*\rb*\sia},{-0.5*\rb*\coa})$)--(v6)--($(v6)+({0.5*\rb*cos(40)},{0.5*\rb*sin(40)})$); \draw ($(v6)+(0.5*\rb*\co,-0.5*\rb*\si)$)--(v6);
\draw ($(v5)+({0.5*\rb*\sia},{-0.5*\rb*\coa})$)--(v5)--($(v5)+({-0.5*\rb*\sib},{-0.5*\rb*\cob})$);
\draw (v2) -- node[below]{$1$} (v4); \draw (v4) -- node[below]{$2$} (v3); \draw (v3) -- node[right]{$3$} (v5); \draw (v5) -- node[right]{$4$} (v1);
\draw (v1) -- node[left]{$5$} (v6); \draw (v6) -- node[left]{$6$} (v2);
\node [black] (v7) at (0,0) [label={[xshift=-0.25cm, yshift=-0.85cm]:$A_1$}, label={[xshift=0.25cm, yshift=0cm]:$A_4$}] {};
\draw (v2)--(v7)--(v4); \draw (v1)--(v7)--(v5);
\draw (v7) to ($(v7)+(0.5*\co,-0.5*\si)$);
\draw [dashed] ($(v2)+({\rb*cos(40)},{\rb*sin(40)})$) to [out=40,in=-100] ($(v1)+({-\rb*\sia},{-\rb*\coa})$);
\node [black] (w1) at (\rx,\ra) {};
\node [black] (w2) at (\rx-\ra*\co,-\ra*\si) {};
\node [black] (w3) at (\rx+\ra*\co,-\ra*\si) {};
\draw (\rx,0) circle (\ra);
\draw (w1)--($(w1)+({-\rb*\sia},{-\rb*\coa})$);
\draw ($(w2)+({\rb*cos(40)},{\rb*sin(40)})$)--(w2);
\draw ($(w3)+({-\rb*cos(40)},{\rb*sin(40)})$)--(w3)--($(w3)+({-\rb*\cob},{\rb*\sib})$);
\node [black] (w6) at ($0.5*(w1)+0.5*(w2)$) {};
\node [black] (w4) at ($0.5*(w2)+0.5*(w3)$) {};
\node [black] (w5) at ($0.5*(w3)+0.5*(w1)$) {};
\draw ($(w4)+({0.5*\rb*\coa},{0.5*\rb*\sia})$)--(w4)--($(w4)+({0.5*\rb*\co},{0.5*\rb*\si})$);
\draw ($(w6)+({-0.5*\rb*\sia},{-0.5*\rb*\coa})$)--(w6)--($(w6)+({0.5*\rb*cos(40)},{0.5*\rb*sin(40)})$); \draw ($(w6)+(0.5*\rb*\co,-0.5*\rb*\si)$)--(w6);
\draw ($(w5)+({0.5*\rb*\sia},{-0.5*\rb*\coa})$)--(w5)--($(w5)+({-0.5*\rb*\sib},{-0.5*\rb*\cob})$);
\draw (w2) -- node[below]{$1$} (w4); \draw (w4) -- node[below]{$2$} (w3); \draw (w3) -- node[right]{$3$} (w5); \draw (w5) -- node[right]{$4$} (w1);
\draw (w1) -- node[left]{$5$} (w6); \draw (w6) -- node[left]{$6$} (w2);
\node [black] (w7) at (\rx,0) [label={[xshift=-0.25cm, yshift=-0.85cm]:$A_1$}, label={[xshift=0.25cm, yshift=0cm]:$A_4$}] {};
\draw (w2)--(w7)--(w4); \draw (w1)--(w7)--(w5);
\draw (w7) to ($(w7)+(-0.5*\co,0.5*\si)$);
\draw [dashed] ($(w4)+({0.5*\rb*\co},{0.5*\rb*\si})$) to [out=30,in=-110] ($(w5)+({-0.5*\rb*\sib},{-0.5*\rb*\cob})$);
\end{tikzpicture}
\caption{Two partial drawings in the case $A_1=A_4$.}
\label{fig:63512notsame}
\end{figure}

In the case $A_1=A_3$, we also have two possible partial drawings depending on the direction of the remaining starting segment at the vertex $A_1=A_3$. Both of them lead to a bad triangle, either $(2,1,1)$ or $(2,1,0)$.

It follows that all the six vertices $A_i, \, i=1, \dots, 6$, are pairwise distinct. Moreover, every pair of vertices $(A_1,A_6), (A_2,A_3)$ and $(A_4,A_5)$ has to be joined by an edge and then the remaining starting segments at these vertices point outside the resulting triangular domains (to avoid bad triangles $(x,0,0), \, x > 0$). We get the drawing shown in Figure~\ref{fig:63512nine}. Denote $U$ the inner nonagonal domain and $c$ its bounding cycle. Note that there are only three remaining vertices to be added to that partial drawing. Denote them $A, B$ and $C$; all three have degree $5$ and lie in $U$. It may happen that two of the starting segments in Figure~\ref{fig:63512nine} lie on the same edge; we call such an edge a \emph{diagonal}.
\begin{figure}[h!]
\pgfmathsetmacro{\ra}{2.5}
\pgfmathsetmacro{\raa}{2}
\pgfmathsetmacro{\rb}{1}
\pgfmathsetmacro{\rc}{0.7}
\pgfmathsetmacro{\rx}{8}
\pgfmathsetmacro{\co}{cos(30)}
\pgfmathsetmacro{\si}{sin(30)}
\pgfmathsetmacro{\coa}{cos(10)}
\pgfmathsetmacro{\sia}{sin(10)}
\pgfmathsetmacro{\cob}{cos(20)}
\pgfmathsetmacro{\sib}{sin(20)}
\tikzstyle{black} = [circle,draw=black,fill=black,ultra thick,inner sep=0pt,minimum size=2mm]
\begin{tikzpicture}
\node [black] (v1) at (0,\ra) {};
\node [black] (v2) at (-\ra*\co,-\ra*\si) {};
\node [black] (v3) at (\ra*\co,-\ra*\si) {};
\draw (0,0) node {$U$} circle (\ra);
\node [black] (v6) at (-\raa*\co,\raa*\si) {};
\node [black] (v4) at (0,-\raa) {};
\node [black] (v5) at (\raa*\co,\raa*\si) {};
\draw (v1)--(v6)--(v2)--(v4)--(v3)--(v5)--(v1);
\node [black] (v11) at ($(v1)+({\rb*\sia},{-\rb*\coa})$) {};
\node [black] (v12) at ($(v1)+({-\rb*\sia},{-\rb*\coa})$) {};
\draw (v11)--(v1)--(v12);
\node [black] (v21) at ($(v2)+({\rb*cos(40)},{\rb*sin(40)})$) {};
\node [black] (v22) at ($(v2)+({\rb*\cob},{\rb*\sib})$) {};
\draw (v21)--(v2)--(v22);
\node [black] (v32) at ($(v3)+({-\rb*cos(40)},{\rb*sin(40)})$) {};
\node [black] (v31) at ($(v3)+({-\rb*\cob},{\rb*\sib})$) {};
\draw (v31)--(v3)--(v32);
\draw ($(v4)+(0,\rc)$)--(v4); %\draw ($(v4)+({-\rc*\cob},{\rc*\sib})$)--(v4)--($(v4)+({\rc*\cob},{\rc*\sib})$);
\draw ($(v6)+(\rc*\co,-\rc*\si)$)--(v6); %\draw ($(v6)+({-\rc*\sia},{-\rc*\coa})$)--(v6)--($(v6)+({\rc*cos(40)},{\rc*sin(40)})$);
\draw ($(v5)+(-\rc*\co,-\rc*\si)$)--(v5); %\draw ($(v5)+({\rc*\sia},{-\rc*\coa})$)--(v5)--($(v5)+({-\rc*cos(40)},{\rc*sin(40)})$);
\draw (v11)--(v12)--(v6)--(v21)--(v22)--(v4)--(v31)--(v32)--(v5)--(v11);
\draw ($(v11)+(0,-\rc)$)--(v11)--($(v11)+(\rc*\sib,-\rc*\cob)$); \draw ($(v12)+(0,-\rc)$)--(v12)--($(v12)+(-\rc*\sib,-\rc*\cob)$);
\draw ($(v21)+(\rc*\co,\rc*\si)$)--(v21)--($(v21)+({\rc*cos(50)},{\rc*sin(50)})$); \draw ($(v22)+(\rc*\co,\rc*\si)$)--(v22)--($(v22)+(\rc*\coa,\rc*\sia)$);
\draw ($(v32)+(-\rc*\co,\rc*\si)$)--(v32)--($(v32)+({-\rc*cos(50)},{\rc*sin(50)})$);
\draw ($(v31)+(-\rc*\co,\rc*\si)$)--(v31)--($(v31)+(-\rc*\coa,\rc*\sia)$);
\end{tikzpicture}
\caption{The nonagonal domain $U$ with starting segments.}
\label{fig:63512nine}
\end{figure}
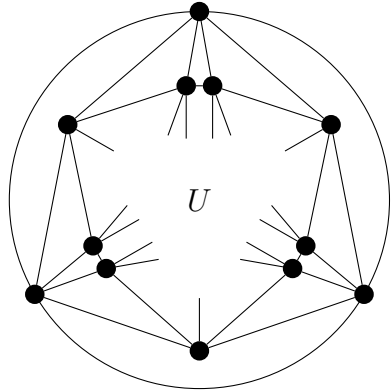
The degree count shows that the number of diagonals $\delta$ equals the number of edges joining the vertices $A,B,C$, so $\delta \in \{0,1,2,3\}$. But if $\delta=0$ then each of the vertices $A,B,C$ is connected to five different vertices on the cycle $c$, and no two such sets of five vertices share more than two elements (otherwise there will be edge crossings). But there are only nine vertices on $c$ and as each of them has no more than two starting segments, we get a contradiction. Similarly, if $\delta=1$, then one of the three vertices $A,B,C$ is connected to five different vertices on $c$, and the other two, to four vertices on $c$, which is again a contradiction. It follows that we have $\delta \in \{2,3\}$ diagonals. They split $U$ into $\delta+1$ domains, with at least two having the property that the intersection of their boundary with $c$ is connected. On the other hand, the vertices $A,B,C$ have $\delta \ge 2$ edges joining them, and so all three lie in a single domain in the complement to the diagonals. It follows that there exists a domain which contains no vertices of the drawing and whose boundary is the union of an arc $c'$ of $c$ and a diagonal. But then the interior of $c'$ contains at least one vertex of $c$ as the graph must be simple, and so there is ``no other endpoint" for a segment starting at that vertex.

This completes the proof of Theorem~\ref{t:bb}.

\vskip.5cm

\emph{Acknowledgements.} This paper is based on the results of the AMSI Vacation Research Scholarship project undertaken by the first author under the supervision of Grant Cairns and the second author. We thank AMSI Vacation Research Scholarship Program for support. We express our deep gratitude to Grant Cairns for his generous contribution to this paper.

%define "extend" and then bad. extend any drawing extended (explain carefully) ; edge which extends the segment
% edge extends or contains starting segment?

\bibliographystyle{amsplain}
\bibliography{bbbib}

\end{document}